\documentclass[12pt]{article}
\usepackage{amsthm,amsmath,amssymb}

\newcommand{\ddd}%
{\text{\raisebox{-2.2pt}{$\cdot\,$}%
 \raisebox{1.7pt}{$\cdot$}%
\raisebox{5.6pt}{$\,\cdot$}}}

\renewcommand{\le}{\leqslant}

\newtheorem{theorem}{Theorem}
\newtheorem{lemma}{Lemma}
\newtheorem{corollary}{Corollary}

\title{Matrices that are
self-congruent only via matrices of
determinant one\footnote{This is the authors' version of a work that was published in \emph{Linear Algebra Appl.} 431 (2009) 1620-1632.}}

\author{Tatyana
G. Gerasimova\\
Institute of Mathematics,
Tereshchenkivska 3\\
Kiev, Ukraine, gerasimova@imath.kiev.ua
  \and
Roger A. Horn
\\Department of Mathematics, University of
Utah\\ Salt Lake City, Utah 84103,
rhorn@math.utah.edu
 \and
Vladimir V. Sergeichuk%
\\
Institute of Mathematics,
Tereshchenkivska 3\\
Kiev, Ukraine, sergeich@imath.kiev.ua}
\date{}
\begin{document}

\maketitle

\begin{abstract}
\raisebox{1pt}{-}\!\!Docovi\'{c} and
Szechtman, [\emph{Proc. Amer. Math.
Soc.} 133 (2005) 2853--2863] considered
a vector space $V$ endowed with a
bilinear form. They proved that all
isometries of $V$ over a field $\mathbb
F$ of characteristic not $2$ have
determinant $1$ if and only if $V$ has
no orthogonal summands of odd dimension
(the case of characteristic $2$ was
also considered). Their proof is based
on Riehm's classification of bilinear
forms. Coakley, Dopico, and Johnson
[\emph{Linear Algebra Appl.} 428 (2008)
796--813] gave another proof of this
criterion over $\mathbb R$ and $\mathbb
C$ using Thompson's canonical pairs of
symmetric and skew-symmetric matrices
for congruence. Let $M$ be the matrix
of the bilinear form on $V$. We give
another proof of this criterion over
$\mathbb F$ using our canonical
matrices for congruence and obtain
necessary and sufficient conditions involving
canonical forms of $M$ for
congruence, of $(M^T,M)$ for
equivalence, and of $M^{-T}M$ (if $M$
is nonsingular) for similarity.


\end{abstract}

\section{Introduction}\label{intr}

Fundamental results obtained by \raisebox{1pt}{-}\!\!Docovi\'{c} and
Szechtman \cite{d-sz} lead to a description of all
$n$-by-$n$ matrices $M$ over any field
$\mathbb F$ such that
\begin{equation}\label{lir1}
\text{$S$ nonsingular and
$S^TMS=M$\quad imply\quad $\det S =
1$}.
\end{equation}

Over a field of characteristic not $2$,
we give another proof of their
description and obtain necessary and
sufficient conditions on $M$ that
ensure \eqref{lir1} and involve
canonical forms of $M$ for congruence,
of $(M^T,M)$ for equivalence, and of
$M^{-T}M$ (if $M$ is nonsingular) for
similarity. Of course, if $\mathbb F$
has characteristic $2$ then every
nonsingular matrix $M$ satisfies
\eqref{lir1}.

A vector space $V$ over $ \mathbb F$
endowed with a bilinear form $B:V\times
V\to\mathbb F$ is called a
\emph{bilinear space}. A linear
bijection ${\cal A}: V\to V$ is called
an \emph{isometry} if
\[
B({\cal A}x,{\cal
A}y)=B(x,y)\qquad\text{for all }x,y\in
V.
\]
If $B$ is given by a matrix $M$, then
the condition \eqref{lir1} ensures that
each isometry has determinant $1$; that
is, the isometry group is contained in
the special linear group.

A bilinear space $V$ is called
\emph{symplectic} if $B$ is a
nondegenerate skew-symmetric form.
It is known that each
isometry of a symplectic space has
determinant $1$ \cite[Theorem 3.25]{Artin}. If $B$ is given by the
matrix
\begin{equation}\label{hys}
Z_{2m}:=\begin{bmatrix} 0&I_m\\
-I_m&0
\end{bmatrix},
\end{equation}
then each isometry is given by a
symplectic matrix (a matrix $S$ is
\emph{symplectic} if
$S^TZ_{2m}S=Z_{2m}$), and so each
symplectic matrix has determinant $1$.

We denote by $M_n(\mathbb F)$ the set
of $n\times n$ matrices over a field
$\mathbb F$ and say that $A,B \in
M_n(\mathbb F)$ are \emph{congruent} if
there is a nonsingular
 $S \in M_n(\mathbb F)$ such that
$S^TAS=B$;  they are \emph{similar} if
$S^{-1}AS=B$ for some nonsingular $S
\in M_n(\mathbb F)$.

The  following theorem is a consequence of
\raisebox{1pt}{-}\!\!Docovi\'{c} and
Szechtman's main theorem \cite[Theorem 4.6]{d-sz}, which is
based on Riehm's  classification of
bilinear forms \cite{riehm}.

\begin{theorem}
\label{thes} Let $M$ be a square matrix
over a field $\mathbb F$ of
characteristic different from $2$. The
following conditions are equivalent:
\begin{itemize}

 \item[\rm(i)]
$M$ satisfies \eqref{lir1}  $($i.e.,
each isometry on the bilinear space
over $\mathbb F$ with scalar product
given by $M$ has determinant $1$$)$,

  \item[\rm(ii)]
$M$ is not congruent to $A\oplus B$
with a square $A$ of odd size.
\end{itemize}
\end{theorem}

\raisebox{1pt}{-}\!\!Docovi\'{c} and
Szechtman \cite{d-sz} also proved that
if $\mathbb F$ consists of more than 2
elements and its characteristic is 2
then $M\in M_n(\mathbb F)$ satisfies
\eqref{lir1} if and only if $M$ is not
congruent to $A\oplus B$ in which $A$
is a singular Jordan block of odd size.
(Clearly, each $M\in M_n(\mathbb F)$
satisfies \eqref{lir1} if $\mathbb F$
has only 2 elements.) Coakley, Dopico,
and Johnson \cite[Corollary
4.10]{c-d-j} gave another proof of
Theorem \ref{thes} for real and complex
matrices only: they used Thompson's
canonical pairs of symmetric and
skew-symmetric matrices for congruence
\cite{thom}. We give another proof of
Theorem \ref{thes} using our canonical
matrices for congruence
\cite{h-s_anyfield,ser_izv}. For the
complex field, pairs of canonical forms of 8
different types are required in
\cite{c-d-j};  our canonical forms are of only
three simple types \eqref{jke}.
Our approach to Theorem \ref{thes} is via canonical forms of matrices;
the approach in \cite{d-sz} is via decompositions of bilinear spaces.

Following \cite{c-d-j}, we denote by
$\Xi_n(\mathbb F)$ the set of all $M
\in M_n(\mathbb F)$  that satisfy
\eqref{lir1}. A computation reveals
that $\Xi_n(\mathbb F)$ is closed under
congruence, that is,
\begin{equation}\label{ohu}
\text{$M \in \Xi_n(\mathbb F)$ and $M$
congruent to $N$ imply $N\in
\Xi_n(\mathbb F)$.}
\end{equation}

The implication
(i)\,$\Rightarrow$\,(ii) of  Theorem
\ref{thes} is easy to establish: let
$M$ be congruent to $N=A\oplus B$, in
which $A \in M_r(\mathbb F)$ and $r$ is
odd. If $S:=(-I_r)\oplus I_{n-r}$, then
$S^TNS=N$ and $\det S=(-1)^r=-1$, and
so $N\notin \Xi_n(\mathbb F)$. It
follows from \eqref{ohu} that $M\notin
\Xi_n(\mathbb F)$.

The implication
(ii)\,$\Rightarrow$\,(i) is not so easy
to establish. It is proved in Section
\ref{sec_pr}. In the rest of this
section and in Section \ref{se_co} we
discuss some consequences of Theorem
\ref{thes}. The first is

\begin{corollary} Let $\mathbb F$ be a field of characteristic not $2$. If $n$ is odd then $\Xi_n(\mathbb F)$ is empty.
$M\in \Xi_2(\mathbb F)$ if and only if
$M$ is not symmetric.
\end{corollary}
Indeed, Theorem \ref{thes} ensures that
$M \notin \Xi_2(\mathbb F)$ if and only
if $M$ is congruent to $[a] \oplus [b]$
for some $a,b \in \mathbb F$, and this
happens if and only if $M$ is
symmetric.

In all matrix pairs that we consider,
both matrices are over $\mathbb F$ and
have the same size. Two matrix pairs
$(A,B)$ and $(C,D)$ are
\emph{equivalent} if there exist
nonsingular matrices $R$ and $S$ over
$\mathbb F$ such that
\[
R(A,B)S:=(RAS,RBS)=(C,D).
\]
A \emph{direct sum} of pairs $(A,B)$
and $(C,D)$ is the pair
\[
(A,B)\oplus(C,D):= (A\oplus C,B\oplus
D)
\]
The \emph{adjoint} of $(A,B)$ is the
pair $(B^T,A^T)$; thus, $(A,B)$ is
\emph{selfadjoint} if $A$ is square and
$A=B^T$. For notational convenience, we
write
\[
M^{-T}:=(M^{-1})^T.\]

We say that $(A,B)$ is a \emph{direct
summand of $(M,N)$ for equivalence} if
$(M,N)$ is equivalent to
$(A,B)\oplus(C,D)$ for some $(C,D)$. A
square matrix $A$ is a \emph{direct
summand of $M$ for congruence}
(respectively, \emph{similarity}) if
$M$ is congruent (respectively,
similar) to $A\oplus B$ for some $B$.

The criterion (ii) in Theorem \ref{thes}
uses the relation of matrix congruence;
one must solve a system of quadratic
equations to check that two matrices
are congruent. The criteria (iii) and
(iv) in the following theorem can be
more convenient to use: one must solve
only a system of linear equations to
check that two matrices are equivalent
or similar. In Section \ref{se_co} we
show that Theorem \ref{thes}
implies

\begin{theorem}
\label{th2} Let $M$ be an $n\times n$
matrix over a field $\mathbb F$ of
characteristic different from $2$. The
following conditions are equivalent:
\begin{itemize}
  \item[\rm(i)]
$M\notin \Xi_n(\mathbb F)$;
  \item[\rm(ii)]
$M$ has a direct summand for congruence
that has odd size;

\item[\rm(iii)]
$(M^{T},M)$ has a direct summand
$(A,B)$ for equivalence, in which $A$
and $B$ are $r\times r$ matrices and
$r$ is odd.

\item[\rm(iv)] $($in
the case of nonsingular $M$$)$
$M^{-T}M$ has a direct summand for
similarity that has odd size.
\end{itemize}
\end{theorem}

For each positive integer $r$, define
the $(r-1)$-by-$r$ matrices
\begin{equation}\label{kut}
  F_r:=\begin{bmatrix}
1&0&&0\\&\ddots&\ddots&\\0&&1&0
\end{bmatrix},\qquad  G_r:=\begin{bmatrix}
0&1&&0\\&\ddots&\ddots&\\0&&0&1
\end{bmatrix},
\end{equation}
and the $r$-by-$r$ matrices
\begin{equation}\label{1aa}
J_r(\lambda ):=\begin{bmatrix}
\lambda&&&0
\\1&\lambda\\
&\ddots&\ddots
\\0&&1&\lambda
\end{bmatrix},
\qquad\Gamma_r :=
\begin{bmatrix}
0&&&&&\ddd
\\&&&&1&
\ddd\\
&&&-1&-1&\\ &&1&1&\\
&-1&-1&
&&\\
1&1&&&&0
\end{bmatrix}.
\end{equation}
Note that
\begin{equation}\label{jov}
\Gamma_{r}^{-T}\Gamma_{r} \quad\text{is
similar to}\quad J_{r}((-1)^{r+1})
\end{equation}
since
\[
\Gamma_{r}^{-T}\Gamma_{r}
=(-1)^{r+1}%
\begin{bmatrix}
\vdots & \vdots & \vdots & \vdots &
\ddd \\
-1 & -1 & -1 & -1 & \\
1 & 1 & 1 &  & \\
-1 & -1 &  &  & \\
1 &  &  &  & 0%
\end{bmatrix}
^{T}\!\!\!\!\cdot
\Gamma_{r}=(-1)^{r+1}%
\begin{bmatrix}
1 & 2 &  & *\\
& 1 & \ddots & \\
&  & \ddots & 2\\
0 &  & & 1
\end{bmatrix}.
\]

Explicit direct summands in the
conditions (ii)--(iv) of Theorem
\ref{th2} are given in the following
theorem.

\begin{theorem}
\label{th3} Let $M$ be an $n\times n$
matrix over a field $\mathbb F$ of
characteristic different from $2$. The
following conditions are equivalent:
\begin{itemize}
  \item[\rm(i)]
$M\notin \Xi_n(\mathbb F)$;
  \item[\rm(ii)]
$M$ has a direct summand for congruence
that is either
\begin{itemize}
  \item
a nonsingular matrix $Q$ such that
$Q^{-T}Q$ is similar to $J_r(1)$  with
odd $r$ $($if $\mathbb F$ is
algebraically closed, then we can take
$Q$ to be $\Gamma_r$ since any such $Q$
is congruent to $\Gamma_r)$, or

  \item
$J_s(0)$ with odd $s$.
\end{itemize}

\item[\rm(iii)]
$(M^T,M)$ has a direct summand for
equivalence that is either
$(I_r,J_r(1))$ with odd $r$, or
$(F_t,G_t)$ with any $t$.

\item[\rm(iv)] $($in
the case of nonsingular $M$$)$
$M^{-T}M$ has a direct summand for
similarity that is $J_r(1)$ with odd
$r$.
\end{itemize}
\end{theorem}

In the following section we deduce
Theorems \ref{th2} and \ref{th3} from
Theorem \ref{thes} and give an algorithm to determine if
$M \in \Xi_n(\mathbb F)$.
In Section \ref{sec_pr} we prove Theorem
\ref{thes}.

\section{Theorem \ref{thes} implies Theorems \ref{th2} and
\ref{th3}} \label{se_co}

Theorem \ref{th3} gives three criteria
for $M\notin\Xi_n(\mathbb F)$ that
involve direct summands of $M$ for
congruence, direct summands of
$(M^T,M)$ for equivalence, and direct
summands of $M^{-T}M$ for similarity.
In this section we deduce these
criteria from Theorem \ref{thes}. For
this purpose, we recall the canonical
form of square matrices $M$ for
congruence over $\mathbb F$ given in
\cite[Theorem 3]{ser_izv}, and derive
canonical forms of selfadjoint pairs
$(M^T,M)$ for equivalence and canonical
forms of cosquares $M^{-T}M$ for
similarity.  Then we establish
conditions on these canonical forms
under which $M\notin\Xi_n(\mathbb F)$.

\subsection{Canonical form
of a square matrix for congruence}
\label{subs1}

Every square matrix $A$ over a field
$\mathbb F$ of characteristic different
from $2$ is similar to a direct sum,
uniquely determined up to permutation
of summands, of {\it Frobenius blocks}
\begin{equation}\label{3}
{\Phi}_{p^l}=\begin{bmatrix} 0&&
0&-c_m\\1&\ddots&&\vdots
\\&\ddots&0&-c_2\\
0&&1& -c_1
\end{bmatrix},
\end{equation}
in which
\[p(x)^l=x^m+c_1 x^{m-1}+\dots+ c_m\]
is an integer power of a polynomial
\begin{equation}\label{star1}
p(x)=x^s+a_1x^{s-1}+\dots+a_s
\end{equation}
that is irreducible over $\mathbb F$.
This direct sum is the \emph{Frobenius
canonical form} of $A$; sometimes it is
called the \emph{rational canonical
form} (see \cite[Section 6]{b-r}).

A Frobenius block has no direct summand
under similarity other than itself,
i.e., it is indecomposable under
similarity. Also, the Frobenius block
${\Phi}_{(x-\lambda)^m}$ is similar to
the Jordan block $J_m(\lambda )$.

If $p(0) = a_s \neq 0$ in (\ref{star1}),
we define
\begin{equation}\label{utz}
p^{\vee}(x):={a}_s^{-1}(1+{a}_1x+
\dots+{a}_sx^s) = p(0)^{-1}x^sp(x^{-1})
\end{equation}
and observe that
\begin{equation}\label{check}
(p(x)^{l})^{\vee} =
p(0)^{-l}x^{sl}p(x^{-1})^{l} =
(p(0)^{-1}x^sp(x^{-1}))^{l} =
(p^{\vee}(x))^{l}.
\end{equation}

The matrix $A^{-T}A$ is the
\emph{cosquare} of a nonsingular matrix
$A$. If two nonsingular matrices are
congruent, then their cosquares are
similar because
\begin{equation}\label{dot}
(S^TAS)^{-T}(S^TAS) =S^{-1}A^{-T}AS.
\end{equation}
If $\Phi$ is a cosquare, we choose a
matrix $A$ such that $A^{-T}A=\Phi$ and
write $\sqrt[T]{\Phi}:=A$ (a {\it
cosquare root} of $\Phi$).

\begin{lemma}\label{new_lemma}
Let $p(x)$ be an irreducible polynomial
of the form \eqref{star1} and let
${\Phi}_{p^l}$ be an $m \times m$
Frobenius block \eqref{3}. Then
\begin{itemize}
\item[\rm(a)]  ${\Phi}_{p^l}$ is a
cosquare if and only if
\begin{equation}\label{eeq2}
p(x)\ne x,\qquad p(x)\ne
x+(-1)^{m+1},\qquad \text{\rm and }
p(x)=p^{\vee}(x).
\end{equation}
\item[\rm(b)] If ${\Phi}_{p^l}$ is a
cosquare and $m$ is odd, then $p(x) =
x-1$.
\end{itemize}
\end{lemma}
\begin{proof}
The conditions in (a) and an explicit
form of $\sqrt[T]{\Phi_{p^l}}$ were
established in \cite[Theorem
7]{ser_izv}; see \cite[Lemma
2.3]{h-s_anyfield} for a more detailed
proof.

(b) By \eqref{eeq2},
$p(x)=p^{\vee}(x)$. Therefore,
$a_s=a_s^{-1}$, so $a_s=\varepsilon
=\pm 1$ and
\[
p(x)=x^{2k+1}+a_1x^{2k}+\dots+
a_kx^{k+1}+a_k\varepsilon
x^{k}+\dots+a_1\varepsilon
x+\varepsilon.
\]
Observe that $p(-\varepsilon )=0$. But
$p(x)$ is irreducible, so  $s=1$ and
$p(x)=x+\varepsilon $. By \eqref{eeq2}
again, $\varepsilon \ne 1$. Therefore,
$p(x) = x-1$.
\end{proof}

Define the {\it skew sum} of two
matrices:
\[
[A\diagdown B]:=
\begin{bmatrix}0&B\\A &0
\end{bmatrix}.
\]

\begin{theorem}\label{th2.1}
Let $M$ be a square matrix over a field
$\mathbb F$ of characteristic different
from $2$. Then
\begin{itemize}
  \item[\rm(a)] $M$ is
congruent to a direct sum of matrices
of the form
\begin{equation}\label{jte}
[\Phi_{p^l} \diagdown I_m], \qquad
Q,\qquad
 J_s(0),
\end{equation}
in which ${\Phi_{p^l}}$ is an $m\times
m$ Frobenius block that is not a
cosquare, $Q$ is nonsingular and
$Q^{-T}Q$ is similar to a Frobenius
block, and $s$ is odd.

  \item[\rm(b)] $M\notin
\Xi_n(\mathbb F)$ if and only if $M$
has a direct summand for congruence
that is either
\begin{itemize}
  \item
a nonsingular matrix $Q$ such that
$Q^{-T}Q$ is similar to $J_r(1)$ with
odd $r$, or
  \item
$J_s(0)$ with odd $s$.
\end{itemize}
\end{itemize}
\end{theorem}

\begin{proof} (a)
This statement is the existence part of
Theorem 3 in \cite{ser_izv} (also
presented in \cite[Theorem
2.2]{h-s_anyfield}), in which a
canonical form of a matrix for
congruence over $\mathbb F$ is given up
to classification of Hermitian forms
over finite extensions of $\mathbb F$.
The canonical block $J_{2m}(0)$ is used
in \cite{ser_izv} instead of $[J_m(0)
\diagdown I_m]$, but the proof of
Theorem 3 in \cite{ser_izv} shows that
these two matrices are congruent.

(b) The ``if'' implication follows
directly from Theorem \ref{thes}. Let
us prove the ``only if'' implication.
If $M \notin \Xi_n(\mathbb F)$, Theorem
\ref{thes} ensures that $M$ is
congruent to $A \oplus B$, in which $A$
is square and has odd size. Part (a)
ensures that $A$ is congruent to a
direct sum of matrices of the form
(\ref{jte}), not all of which have even
size. Thus, $A$ (and hence also $M$)
has a direct summand for congruence
that is either $J_s(0)$ with $s$ odd,
or a nonsingular matrix $Q$ of odd size
such that  $Q^{-T}Q$ is similar to a
Frobenius block ${\Phi}_{p^l}$ of odd
size. Lemma \ref{new_lemma} ensures
that $p(x) = x-1$, so $Q^{-T}Q$ is
similar to $\Phi_{(x-1 )^r}$, which is
similar to $J_r(1)$.
\end{proof}

If $\mathbb F$ is algebraically closed,
then Theorem \ref{th2.1} can be
simplified as follows.

\begin{theorem}\label{th2.1a}
Let $M$ be a square matrix over an
{algebraically closed} field of
characteristic different from $2$. Then
\begin{itemize}
  \item[\rm(a)] $M$ is
congruent to a direct sum of matrices
of the form
\begin{equation}\label{jke}
[J_m(\lambda ) \diagdown I_m], \qquad
\Gamma_r,\qquad
 J_s(0),
\end{equation}
in which $\lambda\ne (-1)^{m+1}$, each
nonzero $\lambda$ is determined up to
replacement by $\lambda^{-1}$,
$\Gamma_r$ is defined in \eqref{1aa},
and $s$ is odd. This direct sum is
uniquely determined by $M$, up to
permutation of summands.

  \item[\rm(b)] $M\notin
\Xi_n(\mathbb F)$ if and only if $M$
has a direct summand for congruence of
the form $\Gamma_r$ with odd $r$ or
$J_s(0)$ with odd $s$.

\end{itemize}
\end{theorem}

\begin{proof}
(a) This canonical form for congruence
was obtained in \cite[Theorem
2.1(a)]{h-s_anyfield}; see also
\cite{hor-ser,hor-ser_can}.

(b) This statement follows from (a) and
Theorem \ref{thes}.
\end{proof}

The equivalence
(i)\,$\Leftrightarrow$\,(ii) in Theorem
\ref{th3} follows from Theorems
\ref{th2.1} and \ref{th2.1a}. (The
equivalence
(i)\,$\Leftrightarrow$\,(ii) in Theorem
\ref{th2} is another form of Theorem
\ref{thes}.)

\subsection{Canonical form
of a selfadjoint matrix pair for
equivalence}\label{subs2}

Kronecker's theorem for matrix pencils
\cite[Chapter 12]{gan} ensures that
each matrix pair $(A,B)$ over $\mathbb
C$ is equivalent to a direct sum of
pairs of the form
\[
(I_m,J_m(\lambda )),\quad
(J_r(0),I_r),\quad (F_s,G_s),\quad
(F_t^T,G_t^T),
\]
in which $F_s$ and $G_s$ are defined in
\eqref{kut}. This direct sum is
uniquely determined by $(A,B)$, up to
permutations of summands. Over a field
$\mathbb F$ of characteristic not $2$,
this canonical form with Frobenius
blocks ${\Phi_{p^l}}$ (see \eqref{3})
instead of Jordan blocks $J_m(\lambda
)$ can be constructed in two steps:
\begin{itemize}
  \item
Use Van Dooren's regularization
algorithm \cite{doo} for matrix pencils
(which was extended to matrices of
cycles of linear mappings in
\cite{vvs2004} and to matrices of
bilinear forms in \cite{h-s_regul}) to
transform $(A,B)$ to an equivalent pair
that is a direct sum of the
\emph{regular part} $(I_k,R)$ with
nonsingular $R$ and canonical pairs of
the form $(J_r(0),I_r)$, $(F_s,G_s)$, and
$(F_t^T,G_t^T)$.

  \item
Reduce $R$ to a direct sum of Frobenius
blocks ${\Phi_{p^l}}$ by a similarity
transformation $S^{-1}RS$; the
corresponding similarity transformation
$ S^{-1}(I_k,R)S=(I_k,S^{-1}RS)$
decomposes the regular part into a
direct sum of canonical blocks
$(I_m,{\Phi}_{p^l})$.
\end{itemize}

\begin{theorem}\label{th2.2}
Let $M$ be a square matrix over a field
$\mathbb F$ of characteristic different
from $2$.

\begin{itemize}
  \item[\rm(a)] The
selfadjoint pair $(M^T,M)$ is
equivalent to a direct sum of
selfadjoint pairs of the form
\begin{equation}\label{jke2}
([I_m\diagdown\Phi_{p^l}^T],
[\Phi_{p^l}\diagdown I_m]), \quad
(\sqrt[T]{\Phi_{q^r}}^T,
\sqrt[T]{\Phi_{q^r}}\,), \quad
(J_s(0)^T,J_s(0)),
\end{equation}
in which ${\Phi_{p^l}}$ is an $m\times
m$ Frobenius block that is not a
cosquare, $\Phi_{q^r}$ is a Frobenius
block that is a cosquare, and $s$ is
odd. This direct sum is uniquely
determined by $M$, up to permutations
of direct summands and replacement, for
each $\Phi_{p^l}$, of any number of
summands of the form
$([I_m\diagdown\Phi_{p^l}^T],
[\Phi_{p^l}\diagdown I_m])$ by
$([I_m\diagdown\Phi_{q^l}^T],
[\Phi_{q^l}\diagdown I_m])$, in which
$q(x):=p^{\vee}(x)$ is defined in
\eqref{utz}.

  \item[\rm(b)] The following
three conditions are equivalent:
\begin{itemize}
  \item[\rm(i)]
$M\notin\Xi_n({\mathbb F})$;
  \item[\rm(ii)]
$(M^T,M)$ has a selfadjoint direct
summand for equivalence of the form
$(\Gamma_r^T, \Gamma_r)$ with odd $r$,
or $(J_s(0)^T,J_s(0))$ with odd $s$;

  \item[\rm(iii)]
$(M^T,M)$ has a direct summand for
equivalence of the form $(I_r,J_r(1))$
with odd $r$, or $(F_t,G_t)$ with any
$t$.
\end{itemize}
\end{itemize}
\end{theorem}

\begin{proof} Let $M$ be a square
matrix over a field $\mathbb F$ of
characteristic different from $2$.

(a) By Theorem \ref{th2.1}(a), $M$ is
congruent to a direct sum $N$ of
matrices of the form \eqref{jte}.
Hence, $(M^T,M)$ is equivalent to
$(N^T,N)$, a direct sum of pairs of the
form \eqref{jke2}.

Uniqueness of this direct sum follows
from the uniqueness assertion in
Kronecker's theorem and the following
four equivalences:

1. $([I_m\diagdown\Phi_{p(x)^l}^T],
[\Phi_{p(x)^l}\diagdown I_m])$ is
equivalent to $(I_m,
\Phi_{p(x)^l})\oplus (I_m,
\Phi_{p^{\vee}(x)^l})$ for each
irreducible polynomial $p(x)\ne x$.

2. $([I_m\diagdown J_m(0)^T],
[J_m(0)\diagdown I_m])$ is equivalent
to $(I_m, J_m(0))\oplus (J_m(0),I_m)$.

3. $(\sqrt[T]{\Phi_{q^r}}^T,
\sqrt[T]{\Phi_{q^r}}\,)$ is equivalent
to $(I,\Phi_{q^r})$.

4. $(J_{2t-1}(0)^T,J_{2t-1}(0))$ is
equivalent to $(F_t^T,G_t^T)\oplus
(G_t,F_t)$.

 To verify the first equivalence, observe that
$(\Phi_{p(x)^l}^{T},I_m)$ is equivalent
to $(I_m,\Phi_{p^{\vee}(x)^l})$ because
\begin{equation}\label{hsi}
\Phi_{p(x)^l}^{-T}\quad \text{is
similar to}\quad \Phi_{p^{\vee}(x)^l}
\end{equation}
for each nonsingular $m\times m$
Frobenius block $\Phi:=\Phi_{p(x)^l}$.
The similarity (\ref{hsi}) follows from
the fact that the characteristic
polynomials of $\Phi^{-T}$ and
$\Phi_{p^{\vee}(x)^l}$ are equal:
\begin{align*}
\chi_{\Phi^{-T}}(x) &=
\det(xI-\Phi^{-1})= \det((-\Phi^{-1})(I
-x\Phi))
\\&=\det(-\Phi^{-1})\cdot
x^m\cdot \det(x^{-1}I
-\Phi)=\chi_{\Phi}^{\vee}(x) =
(p(x)^{l})^{\vee},
\end{align*}
which equals $p^{\vee}(x)^{l}$ by
(\ref{check}).

The second equivalence is obvious.

To verify the third equivalence,
compute
\begin{equation*}
\text{$\sqrt[T]{\Phi_{q^r}}^{-T}(\sqrt[T]{\Phi_{q^r}}^T,
\sqrt[T]{\Phi_{q^r}}\,)I
=(I,\Phi_{q^r})$.}
\end{equation*}

The matrix pairs in the fourth
equivalence are permutationally
equivalent.

(b) ``(i)\,$\Rightarrow$\,(ii)''
Suppose that $M\notin\Xi_n({\mathbb
F})$. By Theorem \ref{th2.1}(b), $M$
has a direct summand  $Q$ for
congruence such that $Q^{-T}Q$ is
similar to $J_r(1)$ with odd $r$, or a
direct summand $J_s(0)$ with odd $s$.
Then $(Q^T, Q)$ or $(J_s(0)^T,J_s(0))$
is a direct summand of $(M^T,M)$ for
equivalence. The pair $(Q^T, Q)$ is
equivalent to $(\Gamma_r^T, \Gamma_r)$
since $Q^{-T}Q$ and $\Gamma_r^{-T}
\Gamma_r$ are similar (they are similar
to $J_r(1)$ by \eqref{jov}) and because
\[
S^{-1}Q^{-T}QS =\Gamma_r^{-T}\Gamma_r
\quad \Longrightarrow\quad
\Gamma_r^TS^{-1}Q^{-T}(Q^T,Q)S
=(\Gamma_r^T,\Gamma_r).
\]

``(ii)\,$\Rightarrow$\,(iii)'' To prove
this implication, observe that
$(\Gamma_r^T, \Gamma_r)$ with odd $r$
is equivalent to $(I_r,\Gamma_r^{-T}
\Gamma_r)$, which is equivalent to
$(I_r,J_r(1))$ by \eqref{jov}, and
\cite[p.\,213]{h-s_anyfield} ensures
that
\begin{equation}\label{jts}
(J_{2t-1}(0)^T,J_{2t-1}(0))\quad
\text{is equivalent to}\quad
(F_t,G_t)\oplus (G_t^T,F_t^T).
\end{equation}

``(iii)\,$\Rightarrow$\,(i)'' Assume
the assertion in (iii). By Theorem
\ref{th2.1}(a), $M$ is congruent to a
direct sum $N=\oplus_i N_i$ of matrices
of the form \eqref{jte}. Then $(M^T,M)$
is equivalent to $(N^T,N)=\oplus_i
(N_i^T,N_i)$. By (iii) and the
uniqueness assertion in Kronecker's
theorem, some $(N_i^T,N_i)$ has a
direct summand for equivalence of the
form $(I_r,J_r(1))$ with odd $r$ or
$(F_t,G_t)$ with any $t$.
\begin{itemize}
  \item
Suppose that the direct summand is
$(I_r,J_r(1))$ with odd $r$. Since
$N_i$ is one of the matrices
\eqref{jte} and $J_r(1)$ with odd $r$
is a cosquare by \eqref{eeq2}, it
follows that $N_i=Q$ and $Q^{-T}Q$ is
similar to $J_r(1)$.

  \item
Suppose that the direct summand is
$(F_t,G_t)$. Since $N_i$ is one of the
matrices \eqref{jte}, \eqref{jts}
ensures that $N_i=J_{2t-1}(0)$.
\end{itemize}
In both the preceding cases, $N_i$ has
odd size, so Theorem \ref{thes} ensures
that $M\notin\Xi_n({\mathbb F})$.
\end{proof}

The equivalences
(i)\,$\Leftrightarrow$\,(iii) in
Theorems \ref{th2} and \ref{th3} follow
from Theorem \ref{th2.2}.

\subsection{Canonical form
of a cosquare for
similarity}\label{subs3}

\begin{theorem}\label{th2.3}
Let $M$ be a nonsingular matrix over a
field $\mathbb F$ of characteristic
different from $2$.

\begin{itemize}
  \item[\rm(a)]  The cosquare
$M^{-T}M$ is similar to a direct sum of
cosquares
\begin{equation}\label{uds}
\Phi_{p^l}\oplus \Phi_{p^l}^{-T},\qquad
\Phi_{q^r},
\end{equation}
in which ${\Phi_{p^l}}$ is a
nonsingular Frobenius block that is not
a cosquare and ${\Phi_{q^r}}$ is a
Frobenius block that is a cosquare.
This direct sum is uniquely determined
by $M$, up to permutation of direct
summands and replacement, for each
$\Phi_{p^l}$, of any number of summands
of the form $\Phi_{p^l}\oplus
\Phi_{p^l}^{-T}$ by $\Phi_{q^l}\oplus
\Phi_{q^l}^{-T}$, in which
$q(x):=p^{\vee}(x)$ is defined in
\eqref{utz}.

  \item[\rm(b)]
$M\notin\Xi_n(\mathbb F)$ if and only
if $M^{-T}M$ has a direct summand for
similarity of the form $J_r(1)$ with
odd $r$.
\end{itemize}
\end{theorem}

\begin{proof}
(a) The existence of this direct sum
follows from Theorem \ref{th2.1}(a)
since $M$ is congruent to a direct sum
of nonsingular matrices $[\Phi_{p^l}
\diagdown I_m]$ and $Q$ (see
\eqref{jte}); the matrices \eqref{uds}
are their cosquares. The uniqueness
assertion follows from uniqueness of
the Frobenius canonical form and
\eqref{hsi}.

(b) By Theorem \ref{th2.2}(b) and
because $M$ is nonsingular,
$M\notin\Xi_n({\mathbb F})$ if and only
if $(M^T,M)$ has a direct summand for
equivalence of the form $(I_r,J_r(1))$
with odd $r$. This implies (b) since
$(M^T,M)$ is equivalent to
$(I_n,M^{-T}M)$.
\end{proof}

The equivalences
(i)\,$\Leftrightarrow$\,(iv) in
Theorems \ref{th2} and \ref{th3} follow
from Theorem \ref{th2.3}.

\subsection{An algorithm}\label{algorithm}

The following simple condition is sufficient to ensure
that $M\in \Xi_n(\mathbb F)$.

\begin{lemma}[{\cite[Theorem 2.3]{c-d-j}}
for $\mathbb {F=R}$ or $\mathbb C$]
\label{theor} Let $\mathbb F$ be a
field of characteristic different from
$2$. If $M \in M_n(\mathbb F)$ and if
its skew-symmetric part $M_w = (M -
M^T)/2$ is nonsingular, then $M\in
\Xi_n(\mathbb F)$.
\end{lemma}

\begin{proof}
Since $M_w$ is skew-symmetric and
nonsingular, there exists a nonsingular
$C$ such that $M_w=C^TZ_{2m}C$, in
which $Z_{2m}$ is defined in
\eqref{hys}. If $S^TMS=M$, then
\[S^TM_wS=M_w,\qquad
(CSC^{-1})^TZ_{2m} (CSC^{-1})=Z_{2m},\]
and so $CSC^{-1}$ is symplectic. By
\cite[Theorem 3.25]{Artin}, $\det
CSC^{-1}=1$, which implies that $\det
S=1$.
\end{proof}

Independent of any condition on $M_w$,
one can use the regularization algorithm
described in \cite{h-s_regul} to reduce
$M$ by a sequence of congruences
(simple row and column operations) to
the form
\begin{equation}\label{star}
B \oplus J_{n_1} (0) \oplus \cdots
\oplus J_{n_p} (0), \text{ $B$
nonsingular and } 1 \le n_1 \le \cdots
\le n_p.
\end{equation}
Of course,
the singular blocks are absent and $B =
M$ if $M$ is nonsingular.

According to Theorem \ref{th2.3}(b), the only
information needed about $B$ in (\ref{star}) is
whether it has any Jordan blocks $J_r(1)$ with odd $r$.
Let $r_k = \operatorname{rank}(B^{-T}B-I)^k$ and set $r_0 = n$.
For each $k=1,\ldots,n$, $B^{-T}B$ has $r_{k-1}-r_k$
blocks $J_j(1)$ of all sizes $j \geq k$ and exactly
$(r_{2k}-r_{2k+1})-(r_{2k+1}-r_{2k+2}) = r_{2k}-2r_{2k+1}+r_{2k+2}$
blocks of the form $J_{2k+1}(1)$ for each $k=0,1,\ldots,[\frac{n-1}{2}]$.

The preceding observations lead to the following algorithm to determine whether
a given $M \in M_n(\mathbb F)$ is in $\Xi_n(\mathbb F)$:
\begin{itemize}
\item[\rm1.]
If $M - M^T$ is nonsingular, then stop: $M \in \Xi_n(\mathbb F)$.
\item[\rm2.]
If $M$ is singular, use the regularization algorithm \cite{h-s_regul} to determine
a direct sum of the form (\ref{star}) to which $M$ is congruent,
and examine the singular block sizes $n_j$.
If any $n_j$ is odd, then stop: $M \notin \Xi_n(\mathbb F)$.
\item[\rm3.]
If $M$ is nonsingular or if all $n_j$ are even, then $M \in \Xi_n(\mathbb F)$
if and only if $r_{2k}-2r_{2k+1}+r_{2k+2}=0$ for all $k=0,1,\ldots,[\frac{n-1}{2}]$.
\end{itemize}

Notice that if $M-M^T$ is nonsingular, then (a) no $n_j$ is odd
since $J_r(0)-J_r(0)^T$ is singular for every odd $r$,
(b) $B-B^T$ is nonsingular, and
(c) $\operatorname{rank}(B^{-T}B-I)=\operatorname{rank}(B^{-T}(B-B^T))=n$,
so $r_k=n$ for all $k=1,2,\ldots$ and
$r_{2k}-2r_{2k+1}+r_{2k+2}=0$ for all $k=0,1,\ldots$.

\section{Proof of Theorem
\ref{thes}}\label{sec_pr}

The implication
(i)\,$\Rightarrow$\,(ii) of Theorem
\ref{thes} was established in Section
\ref{intr}. In this section we prove
the remaining implication
(ii)\,$\Rightarrow$\,(i): we take any
$M \in M_n(\mathbb F)$ that has no
direct summands for congruence of odd
size, and show that $M\in \Xi_n(\mathbb
F)$. We continue to assume, as in
Theorem \ref{thes}, that
 $\mathbb F$ is a
field of characteristic different from
$2$.

By \eqref{ohu} and Theorem
\ref{th2.1}(a), we can suppose that $M$
is a direct sum of matrices of even
sizes of the form $[\Phi_{p^l}
\diagdown I_m]$ and $Q$; see
\eqref{jte}. Rearranging summands, we
represent $M$ in the form
\begin{equation}\label{jyp}
M=M'\oplus M'',\qquad M'\text{ is
}{n'\times n'},\ \ M''\text{ is
}{n''\times n''},
\end{equation}
in which
\begin{itemize}
  \item[($\alpha$)]
$M'$ is the direct sum of all summands
of the form $[\Phi_{(x-1)^m} \diagdown
I_m]$ ($m$ is even by Lemma
\ref{new_lemma}(a)), and
  \item[($\beta$)]
$M''$ is the direct sum of the other
summands; they have the form
$[\Phi_{p^l} \diagdown I_m]$ with
$p(x)\ne x-1$ and $Q$ of even size, in
which ${\Phi_{p^l}}$ is an $m\times m$
Frobenius block that is not a cosquare
and $Q^{-T}Q$ is similar to a Frobenius
block.

\end{itemize}
\medskip

\noindent\emph{Step 1: Show that for
each nonsingular $S$,
\begin{equation}\label{ryo}
S^TMS=M\quad\Longrightarrow \quad
S=S'\oplus S'',\ \ S'\text{ is
}{n'\times n'},\ \ S''\text{ is
}{n''\times n''}.
\end{equation}}

If $S^TMS=M$, then
$S^T(M^T,M)S=(M^T,M)$, and so with
$R:=S^{-T}$ we have
\begin{equation}\label{kit}
(M^T,M)S=R(M^T,M).
\end{equation}
To prove \eqref{ryo}, we prove a more
general assertion: \eqref{kit} implies
that
\begin{equation}\label{geo}
S=S'\oplus S'',\ R= R'\oplus R'',\ \
S',R'\text{ are }{n'\times n'},\ \
S'',R''\text{ are }{n''\times n''}.
\end{equation}
Using Theorem \ref{th2.1a}(a), we
reduce $M'$ and $M''$ in \eqref{jyp} by
congruence transformations over the
algebraic closure $\overline {\mathbb
F}$ of ${\mathbb F}$ to direct sums of
matrices of the form $[J_m(1) \diagdown
I_m]$ and, respectively, of the form
$[J_m(\lambda) \diagdown I_m]$ with
$\lambda \ne 1$ and $\Gamma_r$ with
even $r$. Then
\begin{itemize}
  \item
$(M'^T,M')$ is equivalent over
$\overline {\mathbb F}$ to a direct sum
of pairs of the form
$(I_m,J_m(1))\oplus (J_m(1),I_m)$, and
  \item
$(M''^T,M'')$ is equivalent over
$\overline {\mathbb F}$ to a direct sum
of pairs of the form $(I_m,J_m(\lambda
))\oplus (J_m(\lambda ),I_m)$ with
$1\ne\lambda \in\overline {\mathbb F}$
and $(\Gamma_r^T,\Gamma_r)$ with even
$r$.
\end{itemize}
The pair $(J_m(1),I_m)$ is equivalent
to $(I_m,J_m(1))$. The pair
$(\Gamma_r^T,\Gamma_r)$ is equivalent
to $(I_r,\Gamma_r^{-T}\Gamma_r)$, which
is equivalent to $(I_r,J_r(-1))$ by
\eqref{jov} since $r$ is even. Thus,
\begin{itemize}
    \item[($\alpha'$)]
$(M'^T,M')$ is equivalent to a direct
sum of pairs of that are of the form
$(I_m,J_m(1))$,  and
  \item[($\beta'$)]
$(M''^T,M'')$ is equivalent to a direct
sum of pairs that are either  of the
form $(I_m,J_m(\lambda ))$ with
$\lambda \ne 1$ or of the form
$(J_m(0),I_m)$.
\end{itemize}

We choose $\gamma \in \overline
{\mathbb F}$, $\gamma  \ne -1$, such
that $M''^T+\gamma  M''$ is nonsingular
(if $M''$ is nonsingular, then we may
take $\gamma  = 0$; if $M''$ is
singular, then we may choose any
$\gamma  \ne 0,-1$ such that $(M^T,M)$
has no direct summands of the form
$(I_m,J_m(-\gamma ^{-1})$).

Then \eqref{kit} implies that
\begin{equation*}\label{keit}
(M^T+\gamma  M,M)S=R(M^T+\gamma  M,M).
\end{equation*}

The pair $(M^T+\gamma  M,M)$ is
equivalent to $(I_n,(M^T+\gamma
M)^{-1}M)$, whose Kronecker canonical
pair has the form
\[(I_n,N):=(I_{n'},N')\oplus
(I_{n''},N''),\] in which ($\alpha'$)
and ($\beta'$) ensure that
\begin{itemize}
    \item[($\alpha''$)]
$N'$ (of size $n'\times n'$) is a
direct sum of Jordan blocks with
eigenvalue $(1+\gamma  )^{-1}$, and
    \item[($\beta''$)]
$N''$ (of size $n''\times n''$) is a
direct sum of Jordan blocks with
eigenvalues distinct from $(1+\gamma
)^{-1}$.
\end{itemize}

If $(I_n,N)\tilde S=\tilde R(I_n,N)$,
then $\tilde S=\tilde R$, $N\tilde
S=\tilde SN$, and ($\alpha''$) and
($\beta''$) ensure that $\tilde
S=\tilde S'\oplus \tilde S''$, in which
$\tilde S'$ is $n'\times n'$ and
$\tilde S''$ is $n''\times n''$. Since
$(I_n,N)$ is obtained from $(M^T,M)$ by
transformations within $(M'^T,M')$ and
within $(M''^T,M'')$, \eqref{kit}
implies \eqref{geo}. This proves
\eqref{ryo}.

Since $\det S=\det S'\det S''$, it
remains to prove that
\[
M'\in \Xi_{n'}(\mathbb F),\qquad M''\in
\Xi_{n''}(\mathbb F).
\]

\noindent\emph{Step 2: Show that
$M''\in \Xi_{n''}(\mathbb F)$.}

By Lemma \ref{theor}, it suffices to
show that $2M''_w=M''-M''^T$ is
nonsingular. This assertion is correct
since ($\beta$)  ensures that the
matrix $M''$ is a direct sum of
matrices of the form $[\Phi_{p^l}
\diagdown I_m]$ with $p(x)\ne x-1$ and
$Q$ of even size, and
\begin{itemize}
  \item
for each summand of the form
$[\Phi_{p^l} \diagdown I_m]$,
\[
[\Phi_{p^l} \diagdown
I_m]_w=\begin{bmatrix}0&I_{m}-\Phi_{p^l}^T \\
\Phi_{p^l}-I_m &0
\end{bmatrix}
\]
is nonsingular since $1$ is not an
eigenvalue of $\Phi_{p^l}$;

  \item
for each summand of the form $Q$,
$Q-Q^T=Q^T(Q^{-T}Q-I_r)$ is nonsingular
since $Q^{-T}Q$ is similar to a
Frobenius block ${\Phi}_{p^l}$ of even
size, in which \eqref{eeq2} ensures
that $p(x)\ne x-1$, and so $1$ is not
an eigenvalue of $Q^{-T}Q$.
\end{itemize}

\noindent\emph{Step 3: Show that $M'\in
\Xi_{n'}(\mathbb F)$.}

By $(\alpha )$, $M'$ is a direct sum of
matrices of the form
\begin{equation}\label{tpg}
[\Phi_{(x-1)^m} \diagdown
I_m],\qquad\text{$m$ is even},
\end{equation}
in which ${\Phi_{(x-1)^m}}$ is a
Frobenius block that is not a cosquare;
\eqref{eeq2} ensures that $m$ is even.

Since $C^{-1}{\Phi_{(x-1)^m}}C= J_m(1)$
for some nonsingular $C$, each summand
$[\Phi_{(x-1)^m} \diagdown I_m]$ is
congruent to
\[
\begin{bmatrix}0&I_{m}\\
J_m(1) &0
\end{bmatrix}=
\begin{bmatrix}C^T&0\\
0 &C^{-1}
\end{bmatrix}
\begin{bmatrix}0&I_{m}\\
\Phi_{(x-1)^m} &0
\end{bmatrix}
\begin{bmatrix}C&0\\
0 &C^{-T}
\end{bmatrix},
\]
which is congruent to
\[
\begin{bmatrix}0&
\tilde I_{m}\\
\tilde J_m(1) &0
\end{bmatrix}=
\begin{bmatrix}
\tilde I_{m}&0\\
0 &I_m
\end{bmatrix}
\begin{bmatrix}0&I_{m}\\
J_m(1) &0
\end{bmatrix}
\begin{bmatrix}\tilde
I_{m}&0\\
0 &I_m
\end{bmatrix},
\]
in which
\[
\tilde I_m:=\begin{bmatrix}
0&&1\\
&\ddd&\\
1&&0
\end{bmatrix},\quad
\tilde J_m(1):=\begin{bmatrix}
0&&&1\\
&&1&1\\
&\ddd&\ddd\\
1 &1&&0
\end{bmatrix}
\text{\quad ($m$-by-$m$).}
\]
The matrix $[\tilde J_m(1) \diagdown
\tilde I_m]$ is congruent via a
permutation matrix to
\begin{equation}\label{jos}
\begin{bmatrix}
0&&&K_2\\
&&K_2&L_2\\
&\ddd&\ddd\\
K_2 &L_2&&0
\end{bmatrix},\qquad\text{in which }
K_2:=\begin{bmatrix}
0&1\\
1&0
\end{bmatrix},\
L_2:=\begin{bmatrix}
0&0\\
1&0
\end{bmatrix}.
\end{equation}

We have proved that $[\Phi_{(x-1)^m}
\diagdown I_m]$ is congruent to
\eqref{jos}. Respectively,
\[
[\Phi_{(x-1)^m} \diagdown
I_m]\oplus\dots\oplus [\Phi_{(x-1)^m}
\diagdown I_m]\quad (r\text{ summands})
\]
is congruent to
\begin{equation}\label{jos1}
A_{m,r}:=\begin{bmatrix}
0&&&K_r\\
&&K_r&L_r\\
&\ddd&\ddd\\
K_r &L_r&&0
\end{bmatrix}
\quad (m^2 \text{ blocks}),
\end{equation}
in which
\[
K_r:=\begin{bmatrix}
0&I_r\\
I_r&0
\end{bmatrix}, \qquad
L_r:=\begin{bmatrix}
0&0_r\\
I_r&0
\end{bmatrix}.
\]
Therefore, $M'$ is congruent to some
matrix
\begin{equation*}\label{hps}
N= A_{m_1,r_1}\oplus
A_{m_2,r_2}\oplus\dots\oplus
A_{m_t,r_t},\qquad m_1>m_2>\dots>m_t,
\end{equation*}
in which $r_i$ is the number of
summands $[\Phi_{(x-1)^{m_i}} \diagdown
I_{m_i}]$ of size $2m_i$ in the direct
sum $M'$. In view of \eqref{ohu}, it
suffices to prove that $N\in
\Xi_{n'}(\mathbb F)$.

If
\begin{equation}\label{rxg}
S^TNS=N,
\end{equation}
then (\ref{dot}) implies that
\begin{equation}\label{nht}
N^{-T}NS=SN^{-T}N,
\end{equation}
in which
\begin{equation}\label{hme}
N^{-T}N=\begin{bmatrix}
A_{m_1,r_1}^{-T}A_{m_1,r_1}
&&0\\
&\ddots&\\
0&&A_{m_t,r_t}^{-T}A_{m_t,r_t}
\end{bmatrix}.
\end{equation}
Since
\[
A_{m_i,r_i}^{-1}=
\begin{bmatrix}
*& \dots&*&-
L_{r_i}^T&K_{r_i}\\
\vdots&\ddd&-L_{r_i}^T
&K_{r_i}\\
*&\ddd&K_{r_i}\\
-L_{r_i}^T&\ddd\\
K_{r_i}&&&&0
\end{bmatrix},
\]
we have
\begin{equation}\label{tlw}
A_{m_i,{r_i}}^{-T} A_{m_i,{r_i}}=
\begin{bmatrix}
I_{2r_i}&H_{r_i}&*&\dots&*\\
&I_{2r_i}&H_{r_i}&\ddots&\vdots\\
&&I_{2r_i}&\ddots&*\\
&&&\ddots&H_{r_i}\\
0&&&&I_{2r_i}
\end{bmatrix},\quad
H_{r_i}:=\begin{bmatrix} I_{r_i}&0
\\0&-I_{r_i}
\end{bmatrix};
\end{equation}
the stars denote unspecified blocks.

Partition $S$ in \eqref{nht} into $t^2$
blocks
\begin{equation*}\label{ytl}
S=\begin{bmatrix}
S_{11}&\dots&S_{1t}\\
\vdots&\ddots&\vdots\\
S_{t1}&\dots&S_{tt}\end{bmatrix},\qquad
S_{ij}\ \text{ is }\ 2m_ir_i\times
2m_jr_j,
\end{equation*}
conformally to the partition
\eqref{hme}, then partition each block
$S_{ij}$ into subblocks of size
$2r_i\times 2r_j$ conformally to the
partition \eqref{tlw} of the diagonal
blocks of \eqref{hme}. Equating the
corresponding blocks in the matrix
equation \eqref{nht} (much as in
Gantmacher's description of all
matrices commuting with a Jordan
matrix, \cite[Chapter VIII, \S
2]{gan}), we find that
\begin{itemize}
  \item
all diagonal blocks of $S$ have the
form
\begin{equation*}\label{skf}
S_{ii}=\begin{bmatrix}
C_{i}&&&&*\\
&C_i^H\\
&&\ddots&\\
&&&C_{i}\\
0&&&&C_{i}^H
\end{bmatrix},\qquad
C_{i}^H:=H_{r_i}C_{i}H_{r_i},
\end{equation*}
(the number of diagonal blocks is even
by \eqref{tpg}), and

  \item
all off-diagonal blocks $S_{ij}$ have
the form
\begin{equation*}\label{crk}
\begin{bmatrix}
*&\dots&*\\&\ddots&\vdots\\
&&*\\0
\end{bmatrix}\text{ if
}i<j,\qquad
\begin{bmatrix}
&*&\dots&*\\&&\ddots&\vdots\\
0&&&*
\end{bmatrix}\text{ if
}i>j,
\end{equation*}
in which the stars denote unspecified
subblocks.\footnote{Each Jordan matrix
$J$ is permutation similar to a Weyr
matrix $W_J$ and all matrices commuting
with $W_J$ are block triangular; see
\cite[Section 1.3]{ser}. If we reduce
the matrix \eqref{hme} by simultaneous
permutations of rows and columns to its
Weyr form, then the same permutations
reduce $S$ to block triangular form.}
\end{itemize}

For example, if
\begin{multline*}
N= A_{6,r_1}\oplus A_{4,r_2}\oplus
A_{2,r_3}\\
=\begin{bmatrix}
0&&&&&K_{r_1}\\
&&&&K_{r_1}&L_{r_1}\\
&&&K_{r_1}&L_{r_1}\\
&&K_{r_1}&L_{r_1}\\
&K_{r_1}&L_{r_1}\\
K_{r_1}&L_{r_1}&&&&0\\
\end{bmatrix}\oplus
\begin{bmatrix}
0&&&K_{r_2}\\
&&K_{r_2}&L_{r_2}\\
&K_{r_2}&L_{r_2}\\
K_{r_2}&L_{r_2}&&0\\
\end{bmatrix}\oplus
\begin{bmatrix}
0&K_{r_3}\\
K_{r_3}&L_{r_3}\\
\end{bmatrix},
\end{multline*}
then
\begin{equation*}\label{ppi}
S=\left[\begin{array}{cccccc|cccc|cc}
C_1&*&*&*&*&* &
*&*&*&* &
  *&*\\
&C_1^H&*&*&*&* & &*&*&* &
  &*\\
&&C_1&*&*&* & &&*&* &
  &\\
&&&C_1^H&*&* & &&&* &
  &\\
&&&&C_1&* & &&& &
  &\\
&&&&&C_1^H & &&& &
  &\\
       \hline
&&*&*&*&* & C_2&*&*&*
&*&*\\
&&&*&*&* & &C_2^H&*&*
&&*\\
&&&&*&* & &&C_2&* &&\\
&&&&&* & &&&C_2^H &&\\
       \hline
&&&&*&* & &&*&* &C_3&*\\
&&&&&* & &&&* &&C_3^H
\end{array}\right],
\end{equation*}
in which
\[
C_1^H:=H_{r_1}C_1H_{r_1},\quad
C_2^H:=H_{r_2}C_2H_{r_2},\quad
C_3^H:=H_{r_3}C_3H_{r_3}.
\]

Now focus on equation \eqref{rxg}. The
subblock at the upper right of the
$i$th diagonal block $A_{m_i,r_i}$ of
$N$ is $K_{r_i}$; see \eqref{jos1}. Let
us prove that the corresponding
subblock of $S^TNS$ is
$C_i^TK_{r_i}C_i^H$; that is,
\begin{equation}\label{bpd}
C_i^TK_{r_i}C_i^H=K_{r_i}.
\end{equation}
Multiplying the first horizontal
substrip of the $i$th strip of $S^T$ by
$N$, we obtain
\[
(0\,\dots\,0\,*|\dots|\,
0\,\dots\,0\,*|\, 0\,\dots\,0\;
C_i^TK_{r_i}\,|\,
0\,\dots\,0\,|\dots|\, 0\,\dots\,0);
\]
multiplying it by the last vertical
substrip of the $i$th vertical strip of
$S$, we obtain $C_i^TK_{r_i}C_i^H$,
which proves \eqref{bpd}. Thus, $\det
C_i\det C_i^H=1$. But
\[
\det S=\det C_1\det C_1^H\cdots\det
C_1\det C_1^H\det C_2\det C_2^H\cdots
\]
Therefore, $\det S=1$, which completes
the proof of Theorem \ref{thes}.

\section*{Acknowledgment}

The authors are very grateful to the
referee for valuable remarks and
suggestions, and to  Professor F.
Szechtman for informing us of the
important paper \cite{d-sz}.

\end{document}